\documentclass[3p,times]{elsarticle}

\usepackage{amssymb}
\usepackage[figuresright]{rotating}
\usepackage[]{algorithm2e}

\newtheorem{thm}{Theorem}
\newtheorem{lem}[thm]{Lemma}
\newdefinition{remark}{Remark}
\newproof{proof}{Proof}

\usepackage{graphicx}
\usepackage{amsmath}
\usepackage{epsf,exscale,times}
\usepackage{epsfig}
\usepackage{subfig}
\usepackage{float}
\usepackage{color}
\usepackage{setspace}
\usepackage[toc,page]{appendix}
\usepackage{mathtools}

\allowdisplaybreaks               %% denklem 2. sayfaya tasarsa, ayarl�yor.
\begin{document}
\doublespacing
\begin{frontmatter}

%% Title, authors and addresses

%% use the tnoteref command within \title for footnotes;
%% use the tnotetext command for the associated footnote;
%% use the fnref command within \author or \address for footnotes;
%% use the fntext command for the associated footnote;
%% use the corref command within \author for corresponding author footnotes;
%% use the cortext command for the associated footnote;
%% use the ead command for the email address,
%% and the form \ead[url] for the home page:
%%
%% \title{Title\tnoteref{label1}}
%% \tnotetext[label1]{}
%% \author{Name\corref{cor1}\fnref{label2}}
%% \ead{email address}
%% \ead[url]{home page}
%% \fntext[label2]{}
%% \cortext[cor1]{}
%% \address{Address\fnref{label3}}
%% \fntext[label3]{}

%\dochead{}
%% Use \dochead if there is an article header, e.g. \dochead{Short communication}
%% \dochead can also be used to include a conference title, if directed by the editors
%% e.g. \dochead{17th International Conference on Dynamical Processes in Excited States of Solids}

\title{Multigrid as an exact solver}

%% use optional labels to link authors explicitly to addresses:
%% \author[label1,label2]{<author name>}
%% \address[label1]{<address>}
%% \address[label2]{<address>}

%\author{}
%
%\address{}

\author[cor1,rvt]{Adem Kaya}
\ead{kaya@uni-potsdam.de}
%\author[rvt]{Alfio Borz\`{i}}
%\ead{alfio.borzi@mathematik.uni-wuerzburg.de}

\cortext[cor1]{Corresponding author}

\address[rvt]{Institut f\"{u}r Mathematik, Universit\"{a}t Potsdam
Karl-Liebknecht-Str. 24-25
14476 Potsdam/Golm
Germany}

\begin{abstract}

We provide an alternative Fourier analysis for multigrid  applied to  the Poisson problem in 1D,  based on explicit derivation of spectra of the iteration matrix. 
The new Fourier analysis has advantages over the existing one. It is easy to understand and enables us to write the error equation in terms of the 
eigenvector of the stiffness matrix. When weighted-Jacobi  is used as a smoother with two different 
weights, multigrid is an exact solver.

\end{abstract}

\begin{keyword}
  
   Multigrid, Fourier analysis, smoother, weighted-Jacobi
%% MSC codes here, in the form: \MSC code \sep code
%% or \MSC[2008] code \sep code (2000 is the default)

\end{keyword}

\end{frontmatter}

%%
%% Start line numbering here if you want
%%
%%\linenumbers

%% main text
\section{ Introduction}

We consider the Poisson problem 
with Dirichlet boundary  conditions given by
\begin{eqnarray}
\left\{
  \begin{array}{ll}
  -u^{\prime \prime }(x)  =f(x), \qquad 0<x<1, \\
  u(0)=u(1)=0. 
  \end{array}
\right.
\label{eqn:poisson}
\end{eqnarray}
 The domain of the problem is partitioned into $n-1$ uniform subintervals using the grid points $x_j=jh$ where $h=1/(n-1)$ is the grid size, and $n$ is an odd integer. 
Discretization of the Poisson problem   (\ref{eqn:poisson})  with central finite difference  scheme gives the linear systems of equations,
\begin{displaymath}
A \mathbf{u} =  \mathbf{f}
\end{displaymath}
where  $A =  1/h^2 \text{tridiag}\left(-1, 2, -1 \right) \in \mathbb{R}^{(n-2) \times (n-2)}$. The matrix  $A$ assumes the eigenvectors $\mathbf{v}_k^j = \sin \left(\frac{j k\pi}{n-1}  \right)$ with the 
corresponding eigenvalues $\lambda_k(A) = \frac{4}{h^2}  \sin^2 \left(\frac{k \pi}{2(n-1)}  \right)$ for $k=1, ...,n-2$.
Note that $\mathbf{v}_k^j$ represent the $j-$th entry of the eigenvector $\mathbf{v}_k$.

\section{Smoother}

We use wieghted-Jacobi relaxation as a smoother. Let $a_{i,j}$, $i,j=1,...,n-2$, represent the entries of $A$. We split the matrix $A$ as follows.
\begin{displaymath}
A = D + K 
\end{displaymath}
where $D$ is the diagonal matrix with entries $d_{i,i} = a_{i,i}$, ($i=1,...,n-2$).
Weighted-Jacobi is defined as 
\begin{equation}
\mathbf{x}^{k+1} = (I - \omega D^{-1}A)\mathbf{x}^k + \omega D^{-1}\mathbf{f},
\label{eqn:weightedjacobi}
\end{equation}
where $\omega$ is the weight to be determined. There is no restriction on $\omega$ because  it is not necessary  for a smoother to be 
convergent for a multigrid method to be convergent. Let $R_j^\omega$ represent   the iteration matrix of the weighted-Jacobi given in (\ref{eqn:weightedjacobi}).
\begin{displaymath}
 R_J^\omega\equiv I - \omega D^{-1} A.
\end{displaymath}
We can apply the weighted-Jacobi more than one with different  $\omega$'s to further accelerate the convergence. To this end, we use the following notation.
\begin{displaymath}
 S_J^m \equiv S_J(\omega_1,...,\omega_m) \equiv R_J^{\omega_1}....R_J^{\omega_m},
\end{displaymath}
which means that the weighted-Jacobi is applied $m$ times with the parameters $\omega_i$, $i=1,...,m$. Note that the matrix $S_J^m$ assumes the 
same eigenvectors with the matrix $A$.
 The best way to find the optimal  weights of the weighted-Jacobi method as a smoother, is to do 
   spectral analysis. To this end, we carry out a spectral analysis based on explicit derivation of the spectra of the  iteration matrix of the two-grid.

\section{Other elements of the multigrid method and derivation of the spectrum of the iteration matrix of the two-grid }
In this section,  we introduce interpolation and restriction operators and show some equalities related to them which  are necessary to obtain the 
spectrum of the iteration matrix of the two-grid. We start by setting $A=A^h$, $\mathbf{u}=\mathbf{u}^h$ and $\mathbf{f}=\mathbf{f}^h$ 
where the superscript $h$ which is equivalent to the grid size $h$,  stands for the fine grid.  
The two-grid iteration matrix with  only pre-smoothing  with damped Jacobi relaxation  is given by \cite{Hackbusch1985}
\begin{eqnarray}
  R^{TG} = (I- I_{2h}^h (A^{2h})^{-1} I_h^{2h} A^{h})S_J^m.
\label{eqn:twogridmatrix}
\end{eqnarray}
Our aim is to find the spectrum of $ R^{TG}$. Just for easiness of our analysis, we assumed that $n$ is an odd integer. 
In   Equation (\ref{eqn:twogridmatrix}), 
 the prolongation (interpolation) operator $I_{2h}^h$ is the linear interpolation which has the matrix form
\[
I_{2h}^h=\frac{1}{2}
\begin{bmatrix}
       & 1 &   &  & \\
       & 2 &   &  &   \\
       & 1 & 1 &  & \\
       &   & 2 & \ddots & \\
       &   & 1 & \ddots & \\
       &   &   &    &  1 \\
       &   &   &    &  2 \\
       &   &   &    &  1 
\end{bmatrix} \in \mathbb{R}^{(n-2) \times (n-3)/2}
\]
and  restriction operator $ I_{h}^{2h}$ is the transpose of the prolongation operator 
\begin{eqnarray*}
 I_{h}^{2h} = \left(I_{2h}^{h}  \right)^T.
\end{eqnarray*}
 Coarse grid matrix $A^{2h}$ is defined by Galerkin projection 
\begin{eqnarray}
 A^{2h} = I_{h}^{2h} A^{h} I_{2h}^h.
\label{eqn:galerkinprojection}
\end{eqnarray}
From the above definition, it is easy to show that $A^{2h}$ is also symmetric.
We apply only  pre-smoothing  and do not apply
post-smoothing.

 The prolongation operator $I_{2h}^{h}$ satisfies 
\begin{eqnarray}
 I_{2h}^{h} \mathbf{v}_{k}^{2h} 
  =  \cos^2\left(\frac{k \pi }{2(n-1)}\right)
      \mathbf{v}_k^{h}  
 - \sin^2\left(\frac{k \pi }{2(n-1)}\right)
  \mathbf{v}_{n-1-k}^h,  \qquad 1 \leq k \leq \frac{n-3}{2}
  \label{eqn:prolongation}
\end{eqnarray}
where
\begin{eqnarray}
 \mathbf{v}_{k,j}^{2h} = \sin \left( \frac{2jk\pi}{n-1}\right), \qquad 1\leq j \leq \frac{n-3}{2}.
\label{eqn:coarseeigvector}
 \end{eqnarray}
The restriction operator $I_h^{2h}$  which is the transpose of the prolongation operator   has the following properties. 
\begin{eqnarray}
 I_h^{2h} \mathbf{v}_k^h =2 \cos^2\left(\frac{k\pi }{2(n-1)}\right)\mathbf{v}_{k}^{2h}  \quad \textmd{for} \quad  1 \leq k \leq \frac{n-3}{2}
\label{eqn:restrciton1}
\end{eqnarray}
and
\begin{eqnarray}
 I_h^{2h} \mathbf{v}_{n-1-k}^h =-2 \sin^2\left(\frac{k\pi }{2(n-1)}\right)\mathbf{v}_{k}^{2h}   \quad  \textmd{for} \quad  1 \leq k  \leq \frac{n-3}{2}. 
\label{eqn:restrciton2}
 \end{eqnarray}
As we stated before, the coarse grid matrix is obtained by Galerkin projection. That is,
\begin{eqnarray*}
 A^{2h} = I_{h}^{2h}A^{h} I_{2h}^h.
\end{eqnarray*}
Using this definition and properties of  the restriction and prolongation operators in
(\ref{eqn:prolongation}), (\ref{eqn:restrciton1}) and  (\ref{eqn:restrciton2}), we obtain the spectrum of the coarse matrix $A^{2h}$.
\begin{eqnarray}
 A^{2h} \mathbf{v}_{k}^{2h} = \left(2 \lambda_k(A^{h}) \cos^4(k \pi h/2) +
    2\lambda_{n-1-k}(A^{h}) \sin^4(k\pi h/2 ) \right) \mathbf{v}_{k}^{2h}, \qquad  1 \leq k \leq \frac{n-3}{2}.
\label{eqn:eigcoarse}    
\end{eqnarray}

From the above observations, it is very reasonable to expect that the eigenvectors of the matrix $R^{TG}$ are  linear combinations of 
$\mathbf{v}_k$ and $\mathbf{v}_{n-1-k}$.  We assume that $\mathbf{b}_k=\mathbf{v}_{k} + c \mathbf{v}_{n-1-k}$ is an eigenvector of 
the matrix $R^{TG}$ where $c$ is  to be determined. Imposing $\mathbf{b}_k$ into the definition of $R^{TG}$ in (\ref{eqn:twogridmatrix})
and using the properties of the 
smoother, prolongation, restriction operators and coarse grid matrix we end up with 
\begin{eqnarray*}
 R^{TG}\mathbf{b}_k = \mathbf{v}_k  \left( \lambda_k(S_j^m) - \frac{2 \cos^4(k\pi h/2) \lambda_k(A^h) \lambda_k(S_j^m)}{\lambda_k(A^{2h})} 
                     +  c \frac{2\sin^2(k\pi h/2) \cos^2(k\pi h/2) \lambda_{n-1-k}(A^h) \lambda_{n-1-k}(S_j^m) }{\lambda_k(A^{2h})}   \right) \\
                    + c \mathbf{v}_{n-1-k} \left(   \lambda_{n-1-k}(S_j^m) - \frac{2 \sin^4(k\pi h/2) \lambda_{n-1-k}(A^h) \lambda_{n-1-k}(S_j^m)}{\lambda_k(A^{2h})} 
                     +  \frac{1}{c} \frac{2\sin^2(k\pi h/2) \cos^2(k\pi h/2) \lambda_{k}(A^h) \lambda_{k}(S_j^m) }{\lambda_k(A^{2h})}   \right). 
\end{eqnarray*}
Using   $\lambda_k(A^{2h})$ given in (\ref{eqn:eigcoarse})  and equating the coefficients of $\mathbf{v}_k $ 
and $c \mathbf{v}_{n-1-k}$ in above equation,  we get the following quadratic equation. 
\begin{eqnarray*}
c^2 \left( 2\sin^2(k\pi h/2) \cos^2(k\pi h/2) \lambda_{n-1-k}(A^h) \lambda_{n-1-k}(S_j^m)  \right) \\
        + c \left( 2 \sin^4(k\pi h/2) \lambda_{n-1-k}(A^h) \lambda_{k}(S_j^m)
         -2 \cos^4(k\pi h/2) \lambda_{k}(A^h) \lambda_{n-1-k}(S_j^m) \right)
         \\ -  2 \sin^2(k\pi h/2) \cos^2(k\pi h/2) \lambda_{k}(A^h) \lambda_{k}(S_j^m) =0.
\end{eqnarray*}
Solving the above equation for $c$, we obtain 
\begin{eqnarray}
 c_1 =  \frac{\cos^2(k\pi h/2)\lambda_{k}(A^h)}{\sin^2(k\pi h/2)\lambda_{n-1-k}(A^h)}
\label{eqn:fristrootc}
\end{eqnarray}
and
\begin{eqnarray}
 c_2 = - \frac{\sin^2(k\pi h/2)\lambda_{k}(S_j^m)}{\cos^2(k\pi h/2)\lambda_{n-1-k}(S_j^m)}.
\label{eqn:secondrootc}
\end{eqnarray}
Note that eigenvalues associated to $c_2$ are all zero. More precisely,  
the two-grid iteration matrix $R^{TG}$ assumes 
the eigenvectors
\begin{eqnarray}
\mathbf{b}_k = \left\{ 
 \begin{array}{ll}
\mathbf{v}_k  +  c_1  \mathbf{v}_{n-1-k}, \qquad 1 \leq k \leq \frac{n-1}{2}      \\
 \mathbf{v}_k  + c_2 \mathbf{v}_{n-1-k}, \qquad \frac{n-1}{2} < k \leq n-2
 \end{array}      
 \right.
 \label{eqn:twogrideigenvector}
\end{eqnarray}
with the corresponding eigenvalues 
\begin{eqnarray}
\lambda_k(R^{TG}) = \left\{
 \begin{array}{ll}
 \frac{ \sin^4(k\pi h/2)  \lambda_{n-1-k}(A^h) \lambda_k(\mathbf{S}_j^m)   +  \cos^4(k\pi h/2) \lambda_{k}(A^h) \lambda_{n-1-k}(S_j^m)} {\lambda_{k}(A^{2h})}  , \qquad 1 \leq k \leq \frac{n-1}{2},      \\
 0, \qquad \frac{n-1}{2} < k \leq n-2.
 \end{array}      
 \right.
  \label{eqn:twogrideigenvalue}
\end{eqnarray}

Note that the coarse grid matrix $\mathbf{A}^{2h}$ obtained  by Galerkin projection,
is just a constant multiple  of the 
original matrix which is obtained by rediscretization of the problem (\ref{eqn:poisson}) on coarse grid.
Since $\lambda_k(A^h)=\frac{4}{h^2} \sin^2(k\pi h/2)$ and $\lambda_{n-1-k}(A^h)= \frac{4}{h^2} \cos^2(k\pi h/2)$, Equation (\ref{eqn:eigcoarse}) reduces to
\begin{eqnarray}
 A^{2h} v_{k}^{2h} = \left(\frac{8}{h^2} \sin^2(k\pi h/2) \cos^2(k\pi h/2)  \right) \mathbf{v}_{k}^{2h} = \frac{2}{h^2} \sin^2(k\pi h) \mathbf{v}_{k}^{2h} , \qquad  1 \leq k \leq \frac{n-3}{2}
\label{eqn:eigcoarse2}    
\end{eqnarray}
where $\mathbf{v}_{k}^{2h}$ is given in (\ref{eqn:coarseeigvector}).
Using explicit expressions of  $\lambda_k(A^h)$ and $\lambda_{n-1-k}(A^h)$,   it is easy to show  that $c_1$ given in (\ref{eqn:fristrootc}), is equal to one. 
Hence, in a more compact form, $R^{TG}$ assumes 
the eigenvectors
\begin{eqnarray*}
\mathbf{b}_k = \left\{ 
 \begin{array}{ll}
\mathbf{v}_k  +  \mathbf{v}_{n-1-k}, \qquad 1 \leq k \leq \frac{n-1}{2},      \\
 \mathbf{v}_k  + c_2 \mathbf{v}_{n-1-k}, \qquad \frac{n-1}{2} < k \leq n-2
 \end{array}      
 \right.
\end{eqnarray*}
with the corresponding eigenvalues 
\begin{eqnarray}
\lambda_k(R^{TG}) = \left\{
 \begin{array}{ll}
 \lambda_k(S_j^m) \sin^2(k\pi h/2) +  \lambda_{n-1-k}(S_j^m)  \cos^2(k\pi h/2), \qquad 1 \leq k \leq \frac{n-1}{2},      \\
 0, \qquad \frac{n-1}{2} < k \leq n-2
 \end{array}      
 \right.
 \label{eqn:eigenvalueslaplace}
\end{eqnarray}
where $c_2$ is given in (\ref{eqn:secondrootc}).

If we apply only one pre-smoothing with weighted-Jacobi, that is, for  $S_j^1$,  nonzero eigenvalues of $R^{TG}$ become
\begin{eqnarray*}
 \lambda_k(R^{TG}) = 1- 2\omega \left(\sin^4(k\pi h/2) + \cos^4(k\pi h/2)   \right), \qquad 1\leq k \leq \frac{n-1}{2}.
\end{eqnarray*}
Note that   $\sin^4(k\pi h/2) + \cos^4(k\pi h/2)$ receives its minimum value which is $0.5$ when $k=(n-1)/2$, and its maximum 
value when $k=1$. The maximum value is very close to one, but it is less than one. Assuming that its maximum value is one, we find the optimal 
$\omega$. In order to minimize the spectral radius of $R^{TG}$ we set
\begin{eqnarray*}
 1- 2\omega  \left(\sin^4(k\pi h/2) + \cos^4(k\pi h/2)   \right)|_{k=1} = -1 + 2 \omega  \left(\sin^4(k\pi h/2) + \cos^4(k\pi h/2)   \right)|_{k=(n-1)/2}.
\end{eqnarray*}
Solving above equation (under the assumption $\left(\sin^4(k\pi h/2) + \cos^4(k\pi h/2)   \right)|_{k=1}=1$), we get $\omega=2/3$. This value 
is same with the  value  proposed \cite{multigridtutorial, Hackbusch1985}. The difference is that  our derivation is totally algebraic. 
Furthermore, for $\omega=\frac{2}{3}$,  $\rho(R^{TG})=1/3$. This algebraic derivation also verifies the usability of the following
classification. The Fourier  modes (eigenvectors) $\mathbf{v}_k$ in the range  $1 \leq k <\frac{n-1}{2}$ are called low-frequency or smooth modes and 
the Fourier modes in the range $\frac{n-1}{2} \leq k \leq n-1$ are called high-frequency or oscillatory modes..

We now consider the case $m=2$ which means that pre-smoothing is applied two times with weighted-Jacobi with different $\omega$'s. This case
has not been considered much by the researchers. If more than one pre-smoothing is applied, then the one   which is found
as the optimal  for  one step pre-smoothing, is generally applied. In this case ($m=2$), the nonzero eigenvalues of $R^{TG}$ are given by
\begin{eqnarray}
 \lambda_k(R^{TG}) = 1- 2(\omega_1+\omega_2)  \left(\sin^4(k\pi h/2) + \cos^4(k\pi h/2)   \right) + 4\omega_1 \omega_2 
                     \left(\sin^6(k\pi h/2) + \cos^6(k\pi h/2)   \right), \qquad 1\leq k \leq \frac{n-1}{2}.
\label{eqn:eigenvaluestwoomega}                     
\end{eqnarray}

\begin{lem} \label{lem:findingw}
 The following equality holds 
 \begin{eqnarray*}
 3 (\sin^4(x) + \cos^4(x)) - 2 (\sin^6(x) + \cos^6(x))=1, \qquad \text{for all} \quad   x \in \mathbb{R}.
 \end{eqnarray*}
\end{lem}

\begin{proof}
\begin{eqnarray*}
 3 (\sin^4(x) + \cos^4(x)) - 2 (\sin^6(x) + \cos^6(x))  = 3 (\sin^4(x) + \cos^4(x)) -2 (\sin^2(x) \\
 + \cos^2(x)) (\sin^4(x) -\sin^2(x)\cos^2(x) + \cos^4(x)) = \sin^4(x) + \cos^4(x) +2\sin^2(x)\cos^2(x) \\
 =  (\sin^2(x) + \cos^2(x))^2 =1.
\end{eqnarray*} 
\end{proof}
First, let us observe what happens if we apply weighted-Jacobi with the optimal weight found  for $m=1$,   two times. Substituting $\omega_1=\omega_2=\frac{2}{3}$ 
into (\ref{eqn:eigenvaluestwoomega}), we get $\lambda_k(R^{TG})=1-\frac{8}{9} = 0.\overline{1}$ for all $k$. That is, 
$\rho(R^{TG})=0.\overline{1}$.  Now, we look for different $\omega$'s for which the spectral radius of $R^{TG}$ is reduced further.
By the Lemma \ref{lem:findingw}, for the choices $\omega_1=1$ and $\omega_2=\frac{1}{2}$, the eigenvalues in  (\ref{eqn:eigenvaluestwoomega}) 
become all zero. This means that all eigenvalues of $R^{TG}$ are zero. In other words, two-grid is an exact solver with only one iteration. 
Moreover, since the coarse matrix obtained by Galarkin projection is just a constant multiple of the original matrix on coarse grid,
multigrid is also an exact solver with only 
one iteration.
Eigenvalues of the smoothers $S_j(\frac{2}{3},\frac{2}{3})$ and   $S_j(1,\frac{1}{2})$ for $n=33$ are presented in Figure \ref{fig:eigweightedjacobi}. Although 
for  $S_j(1,\frac{1}{2})$, the two-grid method is an exact solver, we see from Figure \ref{fig:eigweightedjacobi} that corresponding eigenvalues of the 
oscillatory modes are not zero. Furthermore, the maximum eigenvalue  of  $\mathbf{S}_j(1,\frac{1}{2})$ in magnitude in oscillatory region, is
$|\lambda_{21}(S_j(1,\frac{1}{2}))|=0.124581$ which is grater than the maximum eigenvalue   of $S_j(\frac{2}{3},\frac{2}{3})$ in 
magnitude in 
oscillatory region, which is $\lambda_{31}(S_j(1,\frac{1}{2}))=0.\overline{1}$.

\begin{figure}
 \center
\includegraphics[width=10cm]{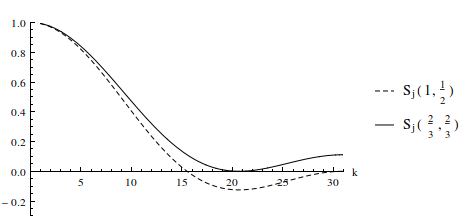}
\caption{Eigenvalues of  weighted-Jacobi  with two steps $S_j^2$  for different $\omega$'s for $n=33$. We assume that the eigenvalues were continuous in $k$.}
\label{fig:eigweightedjacobi}
\end{figure}

\section{The error equation}

Since we have explicit expressions for the eigenvalues of the two-grid iteration matrix $R^{TG}$ in (\ref{eqn:twogrideigenvalue}) and of the
corresponding eigenvectors in (\ref{eqn:twogrideigenvector}), which are linear 
combination of the eigenvectors of $A$, we can see which modes are damped more rapidly. To this end,  we write the error $\mathbf{e}$
in terms of the eigenvectors of $R^{TG}$.
\begin{eqnarray*}
 \mathbf{e} = \sum_{k=1}^{n-2}d_i \mathbf{b}_k = \sum_{k=1}^{(n-1)/2}d_k (\mathbf{v}_k + c_1(k) \mathbf{v}_{n-1-k} ) +  \sum_{k=(n+1)/2}^{n-2}d_k (\mathbf{v}_k + c_2(k) \mathbf{v}_{n-1-k} )
\end{eqnarray*}
where $d_k$ is any constant,  $c_1 = c_1(k)$ and  $c_2 = c_2(k)$ which are given in (\ref{eqn:fristrootc}) and (\ref{eqn:secondrootc}), respectively. 
Since $\lambda_k(R^{TG})=0$ for $k=(n+1)/2...n-2$, after $m$ iterations, the error becomes
\begin{eqnarray*}
 \mathbf{e}^m =   \sum_{k=1}^{(n-1)/2}d_k \lambda_k^m(R^{TG}) (\mathbf{v}_k + c_1(k) \mathbf{v}_{n-1-k} ) = \sum_{k=1}^{(n-1)/2}d_k \lambda_k^m(R^{TG}) \mathbf{v}_k
  + \sum_{k=(n-1)/2}^{n-2}l_k \lambda_{n-1-k}^m(R^{TG})  c_1(n-1-k) \mathbf{v}_{k}  
 \end{eqnarray*}
where $\mathbf{e}^m$ stands for the error after $m$ iterations and $l_k$ is any constant. In above equation on the right, the first sum  
contains the smooth modes and  the second sum contains oscillatory modes.
The first eigenvalue $\lambda_1(\mathbf{R}^{TG})$ is associated with the smoothest 
and the most oscillatory mode. The eigenvalue $\lambda_{(n-1)/2}(R^{TG})$ is associated only with   the eigenvector $\mathbf{v}_{(n-1)/2}$.

\section{Conclusion}

In this work, we provided an alternative Fourier analysis for multigrid  applied to  the Poisson problem in 1D. 
We related multigrid with the exact solver.

\textbf{Note:} \textit{This work is not going to be submitted to any journal. It is free to download and disseminate it}.

\bibliographystyle{plain}
\bibliography{ref}

\appendix

\end{document}